\documentclass[11pt]{article}
\usepackage{amsfonts}
\usepackage{amsmath}

\newcommand{\be}{\begin{equation}}
\newcommand{\ee}{\end{equation}}
\newcommand{\bea}{\begin{eqnarray}}
\newcommand{\eea}{\end{eqnarray}}
\newcommand{\beas}{\begin{eqnarray*}}
\newcommand{\eeas}{\end{eqnarray*}}

\newtheorem{theorem}{Theorem}[section]
\newtheorem{definition}[theorem]{Definition}
\newtheorem{proposition}[theorem]{Proposition}

\newtheorem{lemma}[theorem]{Lemma}
\newtheorem{remark}[theorem]{Remark}
\newtheorem{example}[theorem]{Example}
\newtheorem{examples}[theorem]{Examples}
\newtheorem{foo}[theorem]{Remarks}

\newenvironment{proof}{\addvspace{\medskipamount}\par\noindent{\it Proof}.}
{\unskip\nobreak\hfill$\Box$\par\addvspace{\medskipamount}}

\newcommand{\abs}[1]{\left|#1\right|}     % absolute value

\begin{document}
\title{Uniqueness of solution to scalar BSDEs \\with $L\exp{\left(\mu \sqrt{2\log{(1+L)}}\,\right)}$-integrable terminal values}

\date{}

\author{ Rainer Buckdahn\thanks{Laboratoire de Math\'ematiques, Universit\'e de Bretagne Occidentale, 29285 Brest Cedex, France (Rainer.Buckdahn@univ-brest.fr).}
\and
Ying Hu\thanks{IRMAR,
Universit\'e Rennes 1, Campus de Beaulieu, 35042 Rennes Cedex, France (ying.hu@univ-rennes1.fr)
and School of
Mathematical Sciences, Fudan University, Shanghai 200433, China.
Partially supported by Lebesgue center of mathematics ``Investissements d'avenir"
program - ANR-11-LABX-0020-01,  by ANR CAESARS - ANR-15-CE05-0024 and by ANR MFG - ANR-16-CE40-0015-01.} \and
Shanjian Tang\thanks{Department of Finance and Control Sciences, School of
Mathematical Sciences, Fudan University, Shanghai 200433, China (sjtang@fudan.edu.cn).
Partially supported by National Science Foundation of China (Grant No. 11631004)
and Science and Technology Commission of Shanghai Municipality (Grant No. 14XD1400400).   }}

\maketitle

{\bf Abstract.} In \cite{HuTang18}, the existence of the solution is proved for a scalar linearly growing backward stochastic differential equation (BSDE) if the terminal value is $L\exp{\left(\mu \sqrt{2\log{(1+L)}}\,\right)}$-integrable with the positive parameter $\mu$ being bigger than a critical value $\mu_0$.  In this note, we give the uniqueness result for the preceding BSDE.


\section{Introduction}

Let $\{W_t, t\ge 0\}$ be a standard Brownian motion with values in $\mathbb{R}^d$ defined on some complete probability space $(\Omega,\mathcal{F},\mathbb{P})$, and $\{\mathcal{F}_t, t\ge 0\}$  its natural filtration  augmented by all $\mathbb{P}$-null sets of $\mathcal{F}$. Let us fix a nonnegative real number $T>0$. The $\sigma$-field of predictable subsets of $\Omega \times [0,T]$ is denoted by $\mathcal{P}$.

For any real $p\ge 1$, denote by $L^p$ the set of all $\mathcal{F}_T$-measurable random variables $\eta$ such that $E|\eta|^p<\infty$, by $\mathcal{S}^p$ the set of (equivalent classes of) all real-valued, adapted and c\`adl\`ag processes $\{Y_t, 0\le t\le T\}$ such that
$$||Y||_{\mathcal{S}^p}:=\mathbb{E} \left[\sup_{0 \leq t\leq T} \abs{Y_t}^p \right]^{1/p} < + \infty,$$
by $\mathcal{L}^p$  the set of (equivalent classes of) all real-valued adapted  processes $\{Y_t, 0\le t\le T\}$ such that
$$||Y||_{\mathcal{L}^p}:=\mathbb{E} \left[\int_0^T \abs{Y_t}^p dt\right]^{1/p} < + \infty,$$
and by ${\cal M}^p$ the set of (equivalent classes of) all predictable processes $\{Z_t, 0\le t\le T\}$ with values in $\mathbb{R}^{1 \times d}$ such that
$$||Z||_{{\cal M}^p}:=\mathbb{E}\left[\left(\int_0^T \abs{Z_t}^2 dt \right)^{p/2}\right]^{1/p} < +\infty.$$
%Finally, we will use the notation $Y^*:=\sup_{0\leq t\leq T} \abs{Y_t}$ and we recall that $Y$ belongs to the class (D) as soon as the family $\set{Y_{\tau}: \tau\leq T \textrm{ stopping time}}$ is uniformly integrable.

Consider the following Backward Stochastic Differential Equation (BSDE):
\begin{equation}\label{EDSR}
Y_t=\xi+\int_t^T f(s,Y_s,Z_s)ds-\int_t^T Z_sdW_s, \quad t\in [0,T].
\end{equation}
 Here, $f$ (hereafter called the generator) is a real valued random function defined on the set $\Omega \times [0,T] \times \mathbb{R} \times \mathbb{R}^{1\times d}$, measurable with respect to $\mathcal{P} \otimes \mathcal{B}(\mathbb{R}) \otimes \mathcal{B}(\mathbb{R}^{1\times d})$, and continuous in the last two variables with the following linear growth:
$$|f(s,y,z)-f(s,0,0)|\le \beta |y|+\gamma |z|, \quad (s,y,z)\in [0,T]  \times \mathbb{R} \times \mathbb{R}^{1\times d}$$
with $f_0:=f(\cdot,0,0)\in \mathcal{L}^1,\beta\ge 0$ and $\gamma>0$.  $\xi$ is a real $\mathcal{F}_T$-measurable random variable, and hereafter called the terminal condition or terminal value.

\begin{definition}
By a solution to BSDE (\ref{EDSR}) we mean a pair $\{(Y_t,Z_t), 0\le t\le T\}$ of predictable processes with values in $\mathbb{R} \times \mathbb{R}^{1\times d}$ such that $\mathbb{P}$-a.s., $t \mapsto Y_t$ is continuous, $t \mapsto Z_t$ belongs to $L^2(0,T)$ and $t \mapsto f(t,Y_t,Z_t)$ is integrable, and $\mathbb{P}$-a.s. $(Y,Z)$ verifies  (\ref{EDSR}).
\end{definition}

By BSDE ($\xi$,$f$), we mean the BSDE with generator $f$ and terminal condition $\xi$.

It is well known that for $(\xi, f_0)\in L^p\times \mathcal{L}^p$ (with $p>1$), BSDE (\ref{EDSR}) admits a unique adapted solution $(y,z)$ in the space $\mathcal{S}^p\times \mathcal{M}^p$ if the generator $f$ is  uniformly Lipschitz in the pair of unknown variables. See e.g.
\cite{PP90,EPQ97, BDHPS03} for more details. For $(\xi, f_0)\in L^1\times \mathcal{L}^1$, one needs to restrict the generator $f$ to grow sub-linearly with respect to $z$, i.e.,  with some $q\in [0,1)$,
$$|f(t,y,z)-f_0(t)|\le \beta |y|+\gamma |z|^q,\quad (t,y,z)\in [0,T] \times \mathbb{R} \times \mathbb{R}^{1\times d} $$
for BSDE (\ref{EDSR}) to have a unique adapted solution (see \cite{BDHPS03}) if the generator $f$ is  uniformly Lipschitz in the pair of unknown variables.

In \cite{HuTang18}, the existence of the solution is given for a scalar linearly growing BSDE~\eqref{EDSR} if the terminal value is $L\exp{\left(\mu \sqrt{2\log{(1+L)}}\,\right)}$-integrable with the positive parameter $\mu$ being bigger than a critical value $\mu_0=\gamma \sqrt{T}$, and the  preceding integrability of the terminal value for a positive parameter $\mu$ less than critical value $\mu_0$ is shown to be not sufficient for the existence of a solution.  In this note, we give the uniqueness result for the preceding BSDE under the preceding integrability of the terminal value for $\mu>\mu_0$.

We first establish some interesting properties of the function $\psi(x,\mu)=x\exp{\left(\mu\sqrt{2\log{(1+x)}}\,\right)}$. We observe that the obtained solution $Y$ in \cite{HuTang18} has the nice property:  $\psi(|Y|,a)$ belongs to the class $(D)$ for some $a>0$, which is used to prove the uniqueness of the solution by dividing the whole interval $[0,T]$ into a finite number of sufficiently small subintervals.

%The rest of the paper is organized as follows. Section 2 provides a necessary and sufficient condition for the
%existence of solution to BSDE (\ref{EDSR}) for the typical form of generator $f(t,y,z)=f_0(t)+\beta y+\gamma |z|$, and establishes that the $L\exp{\left(\!\!\sqrt{{2\over \lambda}\log{(1+L)}}\,\right)}$ integrability
% for some $\lambda$ small enough is a sufficient condition for the existence of solution to BSDE (\ref{EDSR}) for  the typical form of the generator $f(t,y,z)=f_0(t)+\beta y+\gamma |z|$. Section 3 is devoted to the  sufficiency of the $L\exp{\left(\!\!\sqrt{{2\over \lambda}\log{(1+L)}}\,\right)}$ integrability condition for
%the existence of solution to BSDE (\ref{EDSR}) of the general linearly growing generator.

\section{Uniqueness}

Define  the function $\psi$:
$$
\psi(x,\mu):=x\exp{\left(\mu\sqrt{2\log{(1+x)}}\,\right)}, \quad (x,\mu)\in [0,+\infty)\times (0, +\infty).
$$
We denote also $\psi(\cdot,\mu)$ as $\psi_\mu(\cdot)$

The following two lemmas can be found in Hu and Tang~\cite{HuTang18}. 

\begin{lemma} \label{duality} For any $x\in \mathbb R$ and $y\ge 0$, we have
\begin{equation}\label{Young0}
e^x y\le e^{\frac{ x^2}{2\mu^2}}+e^{2\mu^2}\psi(y, \mu).
\end{equation}
\end{lemma}

\begin{lemma}\label{Gauss} Let $\mu> \gamma \sqrt{T}$. For any $d$-dimensional adapted process $q$ with $|q_t|\le \gamma$ almost surely, for $t\in [0,T]$,
\begin{equation}
\mathbb E\left[e^{{1\over 2\mu^2}|\int_t^T q_sdW_s|^2} \Bigm |{\cal F}_t\right]\le \frac{1}{\sqrt{1-{\gamma^2\over \mu^2}(T-t)}}.
\end{equation}
\end{lemma}

\begin{proposition} \label{Psi} We have the following assertions on $\psi$:

(i) For $\mu>0$, $\psi(\cdot,\mu)$ is convex. 

(ii) For $c>1$, we have $\psi_\mu(cx)\le \psi_\mu(c)\psi_\mu(x)$ for any $x\ge 0$. 

(iii) For any triple  $(a,b,c)$ with $a>0, b>0$ and $c>0$, we have
$$
\psi(\psi(x,a),b)\le e^{ab^2\over c}\psi(x,a+b+c). 
$$
\end{proposition}

\begin{proof} The first assertion has been shown in \cite{HuTang18}. It remains to show the Assertions (ii) and (iii). 

We prove Assertion (ii). 
\begin{eqnarray*} 
\psi_\mu(cx)&=&cx \exp{\left(\mu\sqrt{2\log{(1+cx)}}\,\right)} \\
&\le &cx \exp{\left(\mu\sqrt{2\log{[(1+c)(1+x)]}}\,\right)}\\
&=& cx \exp{\left(\mu\sqrt{2\log (1+c)+ 2\log (1+x)}\,\right)}\\
&\le& cx \exp{\left(\mu\sqrt{2\log (1+c)}+ \mu\sqrt{2\log (1+x)}\,\right)}\\
&=& \psi_\mu(c)\psi_\mu(x). 
\end{eqnarray*}

We now prove Assertion (iii). 
%With the notation $\psi_\mu:=\psi(\cdot, \mu)$, 
We have 
\begin{eqnarray*}
&& \left(\psi_b \circ \psi_a\right)(x)\\
&=&\psi_a(x)\exp\left({b\sqrt{2\log{(1+\psi_a(x))}}}\right)\\
 &=& x\exp\left({a\sqrt{2\log{(1+x)}}}\right)\exp\left({b\,\sqrt{2\log{\left(1+x e^{a\sqrt{2\log{(1+x)}}}\, \right)}}}\,\right)\\
 &\le&   x\exp\left({a\sqrt{2\log{(1+x)}}}\right)\exp\left({b\,\sqrt{2\log{\left((1+x) e^{a\sqrt{2\log{(1+x)}}}\, \right)}}}\,\right)\\
  &=&   x\exp\left({a\sqrt{2\log{(1+x)}}}\right)\exp\left({b\,\sqrt{2\log{(1+x)}+2a\sqrt{2\log{(1+x)}}}}\,\right).
\end{eqnarray*}
In view of the following elementary inequality:
$$
2a\sqrt{2\log{(1+x)}}\le {a^2b^2\over c^2}+ {2c^2\over b^2}\log{(1+x)}, 
$$
we have 
\begin{eqnarray*}
&& \left(\psi_b \circ \psi_a\right)(x)\\
 &\le&   x\exp\left({a\sqrt{2\log{(1+x)}}}\right)\exp\left({b\,\sqrt{2\log{(1+x)}+{a^2b^2\over c^2}+ {2c^2\over b^2}\log{(1+x)}}}\,\right)\\
 &\le&   x\exp\left({a\sqrt{2\log{(1+x)}}}\right)\exp\left({b\,\sqrt{2\log{(1+x)}}+b\sqrt{{a^2b^2\over c^2}}+ b\sqrt{{2c^2\over b^2}\log{(1+x)}}}\,\right). 
\end{eqnarray*}
Therefore, 
\begin{eqnarray*}
&& \left(\psi_b \circ \psi_a\right)(x)\\
  &\le&   x\exp\left({a\sqrt{2\log{(1+x)}}}\right)\exp\left({b\,\sqrt{2\log{(1+x)}}+{ab^2\over c}+ c\sqrt{2\log{(1+x)}}}\,\right)\\
 &\le&   xe^{{ab^2\over c}}\exp\left({(a+b+c)\sqrt{2\log{(1+x)}}}\right).
\end{eqnarray*}
\end{proof}

Consider the following BSDE:
\begin{equation}\label{bsdez2}
Y_t=\xi+\int_t^T f(s,Y_s,Z_s)\, ds-\int_t^T Z_sdW_s,
\end{equation}
where $f$ satisfies
\begin{equation}\label{lineargrowth}|f(s,y,z)-f_0(s,0,0)|\le \beta |y|+\gamma |z|,
\end{equation}
with $f_0:=f(\cdot,0,0)\in \mathcal{L}^1,\beta\ge 0$ and $\gamma>0$.

\begin{theorem} Let $f$ be a generator which is continuous with respect to $(y,z)$ and verifies inequality (\ref{lineargrowth}), and
$\xi$ be a terminal condition. Let us suppose that there exists $\mu > \gamma \sqrt{T}$ such that $\psi(|\xi|+\int_0^T\!\!\!|f_0(t)|\,dt, \mu)\in L^1(\Omega, \mathbb P).$ 
Then BSDE (\ref{bsdez2}) admits a solution $(Y,Z)$ such that
\begin{eqnarray*}
|Y_t|&\le&\frac{1}{\sqrt{1-{\gamma^2\over \mu^2}(T-t)}}e^{\beta(T-t)}+e^{2\mu^2+\beta (T-t)}\ \mathbb E\left[\psi_\mu\left(|\xi|+\int_t^T\!\!\!|f_0(s)|\,ds\right)  \biggm | {\cal F}_t\right].
\end{eqnarray*}
Furthermore, there exists $a>0$ such that $\psi(Y, a)$ belongs to the class $(D)$. 
\end{theorem}

\begin{proof}
Let us fix $n\in \mathbb N^*$ and $p\in \mathbb N^*$. Set
$$\xi^{n,p}:=\xi^+\wedge n-\xi^-\wedge p, \quad f^{n,p}_0:=f^+_0\wedge n-f^-_0\wedge p, \quad f^{n,p}:=f-f_0+f^{n,p}_0.$$
As the terminal value $\xi^{n,p}$ and $f^{n,p}(\cdot, 0, 0)$ are bounded (hence square-integrable) and $f^{n,p}$ is a continuous generator with a linear growth, in view of the existence result in \cite{LepSan97},
the BSDE $(\xi^{n,p},f^{n,p})$ has a (unique) minimal solution $(Y^{n,p},Z^{n,p})$ in ${\cal S}^2\times {\cal M}^2$.
Set
$$\bar{f}^{n,p}(s,y,z)= |f_0^{n,p}(s)| +\beta y+\gamma |z|, \quad (s,y,z)\in [0,T]\times \mathbb{R}\times \mathbb{R}^{1\times d}.$$
In view of Pardoux and Peng~\cite{PP90},
the BSDE $(|\xi^{n,p}|,\bar{f}^{n,p})$ has a unique solution $(\bar{Y}^{n,p},\bar{Z}^{n,p})$ in ${\cal S}^2\times {\cal M}^2$.

By comparison theorem,
$$|Y^{n,p}_t|\le {\bar Y}^{n,p}_t.$$
Letting $q_s^{n,p}=\gamma \ {\mbox sgn}(Z_s^{n,p})$ and
$$\mathbb P_{q^{n,p}}=\exp\{\int_0^T q_s^{n,p}dW_s-\frac{1}{2}\int_0^T |q_s^{n,p}|^2ds\}\mathbb P, $$
we obtain,
\begin{eqnarray*}
|Y^{n,p}_t|&\le& {\bar Y}^{n,p}_t\\
&=& \mathbb E_{q^{n,p}}\left[e^{\beta(T-t)}|\xi^{n,p}|\Big|{\cal F}_t\right]+\int_t^Te^{\beta(s-t)}|f_0^{n,p}(s)|\, ds\\
&\le& e^{\beta(T-t)}\mathbb E_{q^{n,p}}\left[|\xi^{n,p}|+\int_t^T|f_0^{n,p}(s)|\, ds \Big|{\cal F}_t\right]\\
&\le& \frac{1}{\sqrt{1-{\gamma^2\over \mu^2}(T-t)}}e^{\beta(T-t)}+e^{2\mu^2+\beta (T-t)}\ \mathbb E\left[\psi_\mu\left(|\xi|+\int_t^T\!\!\!|f_0(s)|\,ds\right)  \biggm | {\cal F}_t\right]. 
\end{eqnarray*}
Since $Y^{n,p}$ is nondecreasing in $n$ and non-increasing in $p$, then by the localization method in
\cite{BH06}, there is some $Z\in L^2(0,T; \mathbb R^{1\times d})$ almost surely such that $(Y:=\inf_{p}\sup_{n} Y^{n,p} , Z)$ is an adapted solution.  
Therefore, we have for $a>0$, using Jensen's inequality and the convexity of $\psi_a(\cdot):=\psi(\cdot, a)$ together with Assertion (ii) of Proposition~\ref{Psi}, we have  
\begin{eqnarray*}
\psi_a(|Y^{n,p}_t|)&\le& \psi_a\left(e^{\beta(T-t)}\mathbb E_{q^{n,p}}\left[|\xi^{n,p}|+\int_t^T|f_0^{n,p}(s)|\, ds \Big|{\cal F}_t\right]\right)\\
&\le&  \psi_a\left(e^{\beta(T-t)}\right) \psi_a\left(\mathbb E_{q^{n,p}}\left[|\xi^{n,p}|+\int_t^T|f_0^{n,p}(s)|\, ds\  \Big|{\cal F}_t\right]\right)\\
&\le&  \psi_a\left(e^{\beta(T-t)}\right) \mathbb E_{q^{n,p}}\left[\psi_a\left(|\xi^{n,p}|+\int_t^T|f_0^{n,p}(s)|\, ds \right)\Big|{\cal F}_t\right]\\
&\le &  \psi_a\left(e^{\beta(T-t)}\right) \mathbb E\left[\exp{\left(\int_t^Tq^{n,p}_s\, dW_s\right)}\psi_a\left(|\xi^{n,p}|+\int_t^T|f_0^{n,p}(s)|\, ds \right)\Big|{\cal F}_t\right].
\end{eqnarray*}
For $b>\gamma\sqrt{T}$, applying Lemma~\ref{duality}, we have 
\begin{eqnarray*}
\psi_a(|Y^{n,p}_t|)&\le &  \psi_a\left(e^{\beta(T-t)}\right)  \mathbb E\left[\exp{\left(\int_t^Tq^{n,p}_s\, dW_s\right)}\psi_a\left(|\xi^{n,p}|+\int_t^T|f_0^{n,p}(s)|\, ds \right)\Big|{\cal F}_t\right]\\
&\le& \psi_a\left(e^{\beta(T-t)}\right) \biggl( \mathbb E\left[\exp{\left({1\over 2b^2}\left(\int_t^Tq^{n,p}_s\, dW_s\right)^2\right)}\Big|{\cal F}_t\right]\\
&&+ e^{2b^2} \mathbb E\left[\psi_b\circ\psi_a\left(|\xi^{n,p}|+\int_t^T|f_0^{n,p}(s)|\, ds \right) \Big|{\cal F}_t\right]\biggr). 
\end{eqnarray*}
Using Lemma~\ref{Gauss} and Assertion (iii) of Proposition~\ref{Psi}, we have for any $c>0$, 
\begin{eqnarray*}
&&\psi_a(|Y^{n,p}_t|)\\
&\le & \psi_a\left(e^{\beta(T-t)}\right) \biggl(\frac{1}{\sqrt{1-{\gamma^2\over b^2}(T-t)}} + e^{2b^2+{ab^2\over c}} \mathbb E\left[\psi_{a+b+c}\left(|\xi^{n,p}|+\int_t^T|f_0^{n,p}(s)|\, ds \right) \Big|{\cal F}_t\right]\biggr).
\end{eqnarray*}

For $\mu>\gamma\sqrt{T}$, we can choose $a>0, b>\gamma\sqrt{T}$, and $c>0$ such that $a+b+c=\mu$. Then, we have
\begin{eqnarray*}
&&\psi_a(|Y^{n,p}_t|)\\
&\le & \psi_a\left(e^{\beta(T-t)}\right) \biggl(\frac{1}{\sqrt{1-{\gamma^2\over b^2}(T-t)}} + e^{2b^2+{ab^2\over c}} \mathbb E\left[\psi_\mu\left(|\xi^{n,p}|+\int_t^T|f_0^{n,p}(s)|\, ds \right) \Big|{\cal F}_t\right]\biggr)\\
&\le & \psi_a\left(e^{\beta(T-t)}\right) \biggl(\frac{1}{\sqrt{1-{\gamma^2\over b^2}(T-t)}} + e^{2b^2+{ab^2\over c}} \mathbb E\left[\psi_\mu\left(|\xi|+\int_t^T|f_0(s)|\, ds \right) \Big|{\cal F}_t\right]\biggr). 
\end{eqnarray*}
Letting first $n\to \infty$ and then $p\to \infty$, we have 
\begin{eqnarray*}
&&\psi_a(|Y_t|)\\
&\le & \psi_a\left(e^{\beta(T-t)}\right) \biggl(\frac{1}{\sqrt{1-{\gamma^2\over b^2}(T-t)}} + e^{2b^2+{ab^2\over c}} \mathbb E\left[\psi_\mu\left(|\xi|+\int_t^T|f_0(s)|\, ds \right) \Big|{\cal F}_t\right]\biggr)\\
&\le& \psi_a\left(e^{\beta T}\right) \biggl(\frac{1}{\sqrt{1-{\gamma^2T\over b^2}}} + e^{2b^2+{ab^2\over c}} \mathbb E\left[\psi_\mu\left(|\xi|+\int_0^T|f_0(s)|\, ds \right) \Big|{\cal F}_t\right]\biggr) .
\end{eqnarray*}
Consequently, we have $\psi_a(|Y|)$ belongs to the class $(D)$.
\end{proof}

Now we state our main result of this note. 

\begin{theorem} Assume that the generator $f$ is uniformly Lipschitz in $(y,z)$, i.e.,  there are $\beta>0$ and $\gamma>0$ such that for all  $(y^i,z^i) \in R\times R^{1\times d}$, $i=1,2$, we have
$$
|f(t,y^1,z^1)-f(t,y^2,z^2)|\le \beta |y^1-y^2|+\gamma |z^1-z^2|. 
$$
Furthermore, assume that there exists $\mu > \gamma \sqrt{T}$ such that $\psi(|\xi|+\int_0^T\!\!\!|f(t,0,0)|\,dt, \mu)\in L^1(\Omega, P).$ Then, BSDE (\ref{bsdez2}) admits a unique solution $(Y,Z)$ such that $\psi(Y,a)$ belongs to the class $(D)$ for some $a>0$.
\end{theorem}

\begin{proof} The existence of an adapted solution has been proved in the preceding theorem. It remains to prove the uniqueness. 

For $i=1,2$, let $(Y^i,Z^i)$ be a  solution of BSDE (\ref{bsdez2})  such that $\psi_{a^i}(Y^i)$ belongs to the class $(D)$
for some $a^i>0$. Define 
$$
a:=a^1\wedge a^2, \quad \delta Y:=Y^1-Y^2, \quad  \delta Z:=Z^1-Z^2.
$$
Then both $\psi_{a}(Y^1)$ and $\psi_{a}(Y^2)$ are in the class $(D)$, since $\psi(x,\mu)$ is nondecreasing in $\mu$, and the pair $(\delta Y, \delta Z)$ satisfies the following equation
$$
\delta Y_t=\int_t^T[f(s,Y_s^1,Z_s^1)-f(s,Y_s^2,Z_s^2)]\, ds-\int_t^T\delta Z_s \, dW_s, \quad t\in [0,T]. 
$$
By a standard linearization we see that there exists an adapted  pair of processes $(u,v)$ such that $|u_s|\le \beta,|v_s|\le \gamma,$ and  $f(s,Y_s^1,Z_s^1)-f(s,Y_s^2,Z_s^2)=u_s \delta Y_s+\delta Z_s v_s.$

We define the stopping times
$$
\tau_n:=\inf \{t\ge 0: |Y^1_t|+|Y_t^2|\ge n\}\wedge T, \quad n=1,2,\cdots, 
$$
with the convention that $\inf \emptyset=\infty.$ Since $(\delta Y, \delta Z)$ satisfies the linear BSDE
$$
\delta Y_t=\int_t^T(u_s \delta Y_s+\delta Z_s v_s)\, ds-\int_t^T\delta Z_s \, dW_s, \quad t\in [0,T],
$$
we have the following formula
$$
\delta Y_{t\wedge \tau_n}=\mathbb{E} \left[ e^{\int_{t\wedge \tau_n}^{\tau_n}u_s\, ds+\int_{t\wedge \tau_n}^{\tau_n}\langle v_s, dW_s\rangle-{1\over 2} \int_{t\wedge \tau_n}^{\tau_n}|v_s|^2 ds } \delta Y_{\tau_n}\biggm | {\mathcal F}_t \right]. 
$$
Therefore, 
\begin{eqnarray}\label{deltaY}
|\delta Y_{t\wedge \tau_n}|&\le& \mathbb{E} \left[ e^{\int_{t\wedge \tau_n}^{\tau_n}u_s\, ds+\int_{t\wedge \tau_n}^{\tau_n}\langle v_s, dW_s\rangle } |\delta Y_{\tau_n}| \biggm | {\mathcal F}_t \right]\nonumber \\
&\le&  e^{\beta T}\mathbb{E} \left[ e^{\int_{t\wedge \tau_n}^{\tau_n}\langle v_s , dW_s\rangle } |\delta Y_{\tau_n}| \biggm | {\mathcal F}_t \right]. 
\end{eqnarray}
Now we show that the family of random variables 
$
e^{\int_{t\wedge \tau_n}^{\tau_n}\langle v_s, dW_s\rangle } |\delta Y_{\tau_n}|
$
is uniformly integrable. For this note that, thanks to Lemma \ref{duality},
\begin{eqnarray}\label{ui}
 e^{\int_{t\wedge \tau_n}^{\tau_n}\langle v_s, dW_s\rangle } |\delta Y_{\tau_n}| &\le& e^{{1\over 2a^2} \left(\int_{t\wedge \tau_n}^{\tau_n}\langle v_s, dW_s\rangle\right)^2}+e^{2a^2}\psi_a(|\delta Y_{\tau_n}|).
\end{eqnarray}

For $t\in [T-{a^2\over 4\gamma^2},T]$, we have from Lemma \ref{Gauss},
$$
\mathbb{E} \left[\left| e^{{1\over 2a^2} \left(\int_{t\wedge \tau_n}^{\tau_n}\langle v_s, dW_s\rangle\right)^2}\right|^2\right]=\mathbb{E} \left[ e^{{1\over a^2} \left(\int_{t\wedge \tau_n}^{\tau_n}\langle v_s, dW_s\rangle\right)^2}\right]
\le{1\over \sqrt{1-{2\gamma^2\over a^2}(T-t)}}\le \sqrt{2},
$$
and, thus, the family of random variables $e^{{1\over 2a^2} \left(\int_{t\wedge \tau_n}^{\tau_n}\langle v_s, dW_s\rangle\right)^2}$ is uniformly integrable. 

On the other hand, since $\psi_a$ is nondecreasing and convex, we have thanks to Proposition \ref{Psi} (ii)
\begin{eqnarray*}
\psi_a(|\delta Y_{\tau_n}|)&\le& \psi_a(|Y^1_{\tau_n}|+|Y^2_{\tau_n}|)=\psi_a({1\over 2}\times 2|Y^1_{\tau_n}| +{1\over 2}\times 2|Y^2_{\tau_n}|)\\
&\le&  {1\over 2}\psi_a(2|Y^1_{\tau_n}|) +{1\over 2} \psi_a(2|Y^2_{\tau_n}|)\le  {1\over 2}\psi_a(2)[\psi_a(|Y^1_{\tau_n}|) +\psi_a(|Y^2_{\tau_n}|)]. 
\end{eqnarray*}
From \eqref{ui}
 it now follows that, for $t\in [T-{a^2\over 4\gamma^2},T]$, the family of random variables $
e^{\int_{t\wedge \tau_n}^{\tau_n}\langle v_s\, dW_s\rangle } |\delta Y_{\tau_n}|
$ 
is uniformly integrable. 

Finally, letting $n\to \infty$ in  inequality~\eqref{deltaY}, we have $\delta Y=0$ on the interval $[T-{a^2\over 4\gamma^2},T]$. It is then clear that $\delta Z=0$ on $[T-{a^2\over 4\gamma^2},T]$. The uniqueness of the solution is obtained on the interval $[T-{a^2\over 4\gamma^2},T]$. In an identical way, we have the uniqueness of the solution  on the interval $[T-{a^2\over 2\gamma^2}, T-{a^2\over 4\gamma^2}]$ . By a finite number of steps, we cover in this way the whole interval $[0,T]$, and we conclude the uniqueness of the solution  on the interval $[0,T]$. 
\end{proof}


\begin{thebibliography}{99}
\bibitem{BDHPS03} P. Briand, B. Delyon, Y. Hu, E. Pardoux and L. Stoica, $L^p$ solutions of backward stochastic differential equations. {\it Stochastic Process. Appl.} {\bf 108} (2003), no. 1, 109--129.

\bibitem{BH06}  P. Briand and Y. Hu, BSDE with quadratic growth and unbounded terminal value. {\it Probab. Theory Related Fields} {\bf 136} (2006), no. 4, 604--618.
%
%\bibitem{DHR11}   F. Delbaen, Y.  Hu and A. Richou, On the uniqueness of solutions to quadratic BSDEs with convex generators and unbounded terminal conditions. {\it Ann. Inst. Henri Poincar\'e Probab. Stat.} {\bf 47} (2011), no. 2, 559--574.

\bibitem{EPQ97} N. El Karoui, S. Peng and M. C. Quenez, Backward stochastic differential equations in finance. {\it Math. Finance} {\bf 7} (1997), no. 1, 1--71.

\bibitem{HuTang18} Y. Hu and S. Tang,  Existence of solution to scalar BSDEs with $L\exp{\left(\!\!\sqrt{{2\over \lambda}\log{(1+L)}}\,\right)}$-integrable terminal values.  {\it Electron. Commun. Probab.} {\bf 23} (2018), Paper No. 27, 11pp.

\bibitem{LepSan97} J. P. Lepeltier and J. San Martin, Backward stochastic differential equations with continuous coefficient. {\it Statist. Probab. Lett.} {\bf 32} (1997),  425--430.

\bibitem{PP90}  E. Pardoux and S. Peng, Adapted solution of a backward stochastic differential equation. {\it Systems Control Lett.} {\bf 14} (1990), no. 1, 55--61.
%
%\bibitem{T06} S. Tang,  Dual representation as stochastic differential games of backward stochastic differential equations and dynamic evaluations. {\it C. R. Math. Acad. Sci. Paris} {\bf 342} (2006), no. 10, 773--778.
\end{thebibliography}
\end{document}